\documentclass[10pt]{amsart}

\usepackage{graphicx}
\usepackage{amssymb,mathrsfs,latexsym}

\usepackage{enumerate}        

\usepackage[all,dvips]{xy}    
\UseComputerModernTips        

\newtheorem{theorem}{Theorem}[section]
\newtheorem{lemma}[theorem]{Lemma}

\newtheorem{proposition}[theorem]{Proposition}

\newtheorem{conjecture}[theorem]{Conjecture}

\theoremstyle{definition}
\newtheorem{definition}[theorem]{Definition}
\newtheorem{example}[theorem]{Example}

\newtheorem{note}[theorem]{Note}
 
\begin{document}

\title[Intermediate Growth]
{Iterated monodromy groups of intermediate growth}

\author[Dougherty]{Ashley S. Dougherty}
\address{Department of Education, Department of Mathematics, Kutztown University,  
Kutztown, PA 19530}
\email{gffadougherty@yahoo.com}

\author[Kindelin]{Lydia R. Kindelin}
\address{Department of Mathematics, University of Dayton, 
Dayton, OH 45302}
\email{kindelin1@notes.udayton.edu}

\author[Reaves]{Aaron M. Reaves}
\address{Department of Mathematics, Morehouse College, 
Atlanta, GA 30314}
\email{reavesam@gmail.com}

\author[Walker]{Andrew J. Walker}
\address{Department of Mathematics, University of California, 
Riverside, CA 92521}
\email{walker@math.ucr.edu}

\author[Zakahi]{Nathaniel F. Zakahi}
\address{Department of Mathematics and Computer Science, Denison University,
Granville, OH 43023}
\email{zakahi\_n@denison.edu}

\begin{abstract}
We give new examples of groups of intermediate growth, by a method that was first developed by Grigorchuk and later adapted by Bux and P\'{e}rez.
Our examples are the groups generated by the automata with the kneading sequences
$11(0)^{\omega}$ and $0(011)^{\omega}$. By results of Nekrashevych, both of these groups are iterated monodromy groups
of complex post-critically finite quadratic polynomials. 

We include a complete, systematic description of Bux and P\'{e}rez's adaptation of Grigorchuk's method. We also prove, as a sample application of this method, 
that the groups
determined by the automata with kneading sequence 
$1(0^{k})^{\omega}$ ($k \geq 2$) 
have intermediate growth, although this
result is implicit in a survey article by Bartholdi, Grigorchuk, and Sunik.

The paper concludes with an example of a group with no admissible length function; i.e., the group in question admits no length 
function with the properties required by the arguments of Bux and P\'{e}rez. Whether the group has intermediate growth appears to be an 
open question.
\end{abstract}



\maketitle

\section{Introduction} \label{section:intro}

Let $p: \mathbb{C} \rightarrow \mathbb{C}$ be a complex polynomial. We say that $p$ is \emph{post-critically finite} if, for each
critical point $c$, the set of all forward iterates $\{ p(p(p \ldots (c))) \}$ of $c$ is a finite set. Nekrashevych \cite{N} has shown how to 
associate a group, called an \emph{iterated monodromy group}, to any post-critically finite complex polynomial. The iterated monodromy group
of $p$, denoted $\mathrm{IMG}(p)$, acts on an infinite rooted $n$-ary tree if the degree of $p$ is $n$.

We first began this project because of our interest in the following conjecture from \cite{BP}, where it is attributed to Nekrashevych:
\begin{conjecture} \label{conj:Nek}
If $p: \mathbb{C} \rightarrow \mathbb{C}$ is a post-critically finite quadratic polynomial with pre-periodic kneading
sequence, then $\mathrm{IMG}(p)$ has intermediate growth.
\end{conjecture}
The first positive evidence was obtained by Bux and P\'{e}rez \cite{BP}, who showed that 
$\mathrm{IMG}( z^{2} + i)$ has subexponential growth.
(The proof that $\mathrm{IMG}(z^{2} + i)$ also has superpolynomial growth is comparatively straightforward, so their work
proves that $\mathrm{IMG}(z^{2} + i)$ has intermediate growth.) There are known counterexamples, however. 
Grigorchuk and Zuk \cite{GZ} showed that 
$\mathrm{IMG}(z^{2}-1)$ has exponential growth. The tuning \cite{Dou87} of $z^{2} - 1$ by $z^{2}+i$ results in a post-critically finite
quadratic polynomial $g(z) = z^{2} + c$
with pre-periodic kneading sequence such that $\mathrm{IMG}(z^{2}-1)$ embeds in $\mathrm{IMG}(g)$. It follows easily that $\mathrm{IMG}(g)$
also has exponential growth, making it a counterexample to Conjecture \ref{conj:Nek}. (Note that $z^{2}-1$ has a periodic kneading sequence, so
it is not a counterexample in itself.) The following conjecture appears to be open:
\begin{conjecture} \label{conj:improvedNek}
If $p: \mathbb{C} \rightarrow \mathbb{C}$ is a non-renormalizable post-critically finite quadratic polynomial with pre-periodic kneading
sequence, then $\mathrm{IMG}(p)$ has intermediate growth.
\end{conjecture}
The hypothesis of non-renormalizability rules out the counterexamples to Conjecture \ref{conj:Nek} that arise from tuning.

Our goal here is to give two more examples in support of the latter conjecture, namely the iterated monodromy groups of 
polynomials with the kneading sequences $11(0)^{\omega}$ and $0(011)^{\omega}$. (The kneading sequence of a kneading automaton
is described in Definition \ref{def:kneadingsequence}; 
Theorem \ref{thm:kneadingsequence} says that the latter definition agrees with the classical definition of the kneading sequence of
a polynomial up to relabeling.) 
We also give a short proof that the groups generated by the automata with the 
kneading sequences of the form $1(0^{k})^{\omega}$ have intermediate growth, although a proof that these groups have intermediate growth
can be obtained from Theorem 10.5 of \cite{BGS}. A secondary goal is to provide an exposition of the methods of Bux and P\'{e}rez. While formulated differently, the key idea in their arguments, Proposition 10 of \cite{BP}, is based upon the work of Grigorchuk in his proof that the First Grigorchuk group has subexponential growth \cite{FG}. We attempt to isolate
the precise hypotheses that are necessary to make their arguments work, and state general theorems. The main result in this direction 
is Theorem \ref{thm:intgrowth}, which gives a simple sufficient condition for the group of a kneading automaton over an alphabet with two letters
to have subexponential growth. 

The paper is structured as follows. In Section \ref{section:background}, 
we review the definition of automata, and explain how an automaton can be used to define
a group that acts by automorphisms on a rooted tree. Section \ref{section:buxperez} 
contains an exposition of Bux and P\'{e}rez's formulation of Grigorchuk's method. 
Section \ref{section:main} contains proofs that the groups determined by the automata with the kneading sequences $1(0^{k})^{\omega}$
($k \geq 1$), $11(0)^{\omega}$, and $0(011)^{\omega}$ have intermediate growth. In Section \ref{section:final}, 
we give an example of a group defined by an automaton with pre-periodic kneading sequence to 
which this method does not apply. (Specifically, the group in question 
has no admissible length function -- see Definition \ref{def:admissible}.) 

All of the results in Sections \ref{section:main} and \ref{section:final} 
were proved by the authors at SUMSRI, an REU program based at Miami University, 
during the summer of 2011. The authors also produced an article as part of the REU, which can be found at 
www.units.muohio.edu/sumsri/sumj/2011/fp\_alg.pdf . 
The material in Section \ref{section:buxperez} was contributed by Daniel Farley. 

The authors would like to thank Rodrigo P\'{e}rez for clarifying the status of Conjecture \ref{conj:Nek} to us.

\section{Background} \label{section:background}

Essentially all of the material in this section has appeared in \cite{N}. We gather it here (sometimes in slightly altered form)
for the reader's convenience.

\subsection{Automata, Moore Diagrams, Trees} \label{subsection:amt}

\begin{definition}
Let $X$ be an alphabet. An \emph{automaton} $A$ over $X$ is given by:
\begin{enumerate}
\item a set of \emph{states}, usually also denoted $A$;
\item a map $\tau: A \times X \rightarrow X \times A$.
\end{enumerate}
For each state $a \in A$, we define a function $\tau_{a}: X \rightarrow X$
by the rule $\tau_{a}(x) = \pi_{1}(\tau(a,x))$, where $\pi_{1}: X \times A \rightarrow X$
is projection on the first coordinate. If $\tau_{a} \neq \mathrm{id}_{X}$, 
we say that $a$ is an \emph{active state}.
We say that the automaton $A$ is \emph{invertible} 
if each $\tau_{a}: X \rightarrow X$ is a bijection. 
\end{definition}

Automata can be conveniently described using Moore diagrams.

\begin{definition}
Let $A$ be an automaton over the alphabet $X$. The \emph{Moore diagram} for $A$ is a directed
labelled graph $\Gamma$, defined as follows. The vertices of $\Gamma$ are the states of
$A$. If $a,b \in A$ and $\tau(a,x) = (y,b)$, then there is a directed edge from $a$ to $b$ 
with the label $(x,y)$.
\end{definition}

\begin{example} \label{example:moore}
We define an automaton $A$ as follows. The states are $a$, $b$, $t$, and $id$. The alphabet 
is $X = \{ 0, 1 \}$. We define the function $\tau: A \times X \rightarrow X \times A$ by the
rule:

\vspace{10pt}

\begin{center}
\begin{tabular}{lll}
$\tau(a,0) = (0,id)$ & \quad & $\tau(b,0) = (0,b)$ \\
$\tau(a,1) = (1,t)$ & \quad & $\tau(b,1) = (1,a)$ \\
$\tau(t,0) = (1,id)$ & \quad & $\tau(id,0) = (0,id)$ \\
$\tau(t,1) = (0,id)$ & \quad & $\tau(id,1) = (1,id)$ 
\end{tabular}
\end{center}

\vspace{10pt}

The Moore diagram for this automaton is pictured in Figure \ref{figure:1}.

\begin{figure} [!h] 
\begin{center}
\includegraphics{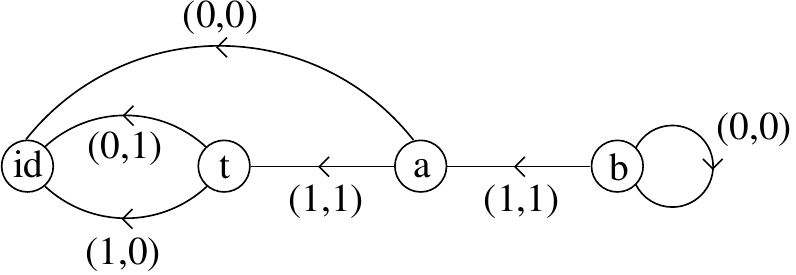}
\end{center}
\caption{The Moore diagram for the automaton from Example \ref{example:moore}.}
\label{figure:1}
\end{figure}

\end{example}

\begin{definition}
Let $X$ be an alphabet. We let $X^{\ast}$ denote the free monoid generated by $X$. Thus, $X^{\ast}$ is the 
collection of all positive strings in the letters of $X$, including the empty string. We can associate
to $X^{\ast}$ a tree, which we also denote $X^{\ast}$, as follows: the vertices of the tree are the 
members of $X^{\ast}$, and there is an edge connecting two vertices $w_{1}, w_{2} \in X^{\ast}$ if and only
if $w_{2} = w_{1}x$, for some $x \in X$ (or $w_{1} = w_{2}x$, for some $x \in X$).
\end{definition}

It is easy to see that if $|X| = n$, then $X^{\ast}$ is a complete rooted $n$-ary tree, i.e., there
is a root vertex, $\emptyset$, of degree $n$, and all other vertices have degree $n+1$.

\subsection{The group determined by an automaton} \label{subsection:ga}

We will now explain how an invertible automaton over $X$ determines
a group of automorphisms of the tree $X^{\ast}$.

\begin{definition}
Let $A$ be an automaton over $X$. Let $a \in A$ and $x \in X$. We define
$a_{\mid x} = \pi_{2} \tau(a,x)$, where $\pi_{2}: X \times A \rightarrow A$
is projection on the second factor.
\end{definition}

For each $a \in A$, we can define a function $a: X^{\ast} \rightarrow X^{\ast}$ by
induction on the length of a word $w \in X^{\ast}$. If $|w| = 0$ (so $w$ is the null string),
then we set $a(w) = w$. If $|w| > 0$, then we can write $w=xw_{1}$, for some $x \in X$ and $w_{1} \in X^{\ast}$.
We define $a(w) = \tau_{a}(x) a_{\mid x}(w_{1})$.

It is easy to see that the function $a: X^{\ast} \rightarrow X^{\ast}$ preserves parents and children: if
$w_{2} = w_{1}x$ for some $x \in X$ (so $w_{2}$ is a child of $w_{1}$), then $a(w_{2})$ is also a child
of $a(w_{1})$. If the automaton $A$ is invertible then $a: X^{\ast} \rightarrow X^{\ast}$ is an automorphism
for each $a \in A$ (\cite{N}; pg. 7).

\begin{definition}
Let $A$ be an invertible automaton. The group defined by $A$, $G(A)$, is the subgroup of $\mathrm{Aut}(X^{\ast})$ 
generated by:
$$ \{ a: X^{\ast} \rightarrow X^{\ast} \mid a \in A \}.$$
\end{definition}

While our definition of $a_{\mid x}$ (for $a \in A$ and $x \in X$) is good enough to define the action 
of $G(A)$ on its associated tree, we will often need a definition of $w_{\mid x}$, where $w \in A^{\ast}$
and $x \in X$.

\begin{definition}
Let $a_{1}a_{2}\ldots a_{n} \in A^{\ast}$. For $x \in X$, we define $(a_{1}a_{2}\ldots a_{n})_{\mid x}$
in $A^{\ast}$ by the rule:

$$ (a_{1}a_{2}\ldots a_{n})_{\mid x} = a_{1 \mid \tau_{a_{1}}\tau_{a_{2}}\ldots \tau_{a_{n}}(x)} \ldots
a_{n-1 \mid \tau_{a_{n}}(x)}a_{n \mid x}.$$
\end{definition}

\begin{note}
An easy way to compute $(a_{1}a_{2}\ldots a_{n})_{\mid x}$ is as follows.
Form the concatenation $a_{1}a_{2}\ldots a_{n} x$ (in $(A \cup X)^{\ast}$)
and then regard the function $\tau: A \times X \rightarrow X \times A$ as a rule
telling us to replace a substring $\widehat{a} \widehat{x}$
($\widehat{a} \in A$, $\widehat{x} \in X$) by the string $\widetilde{x} \widetilde{a}$,
where $\tau(\widehat{a}, \widehat{x}) = (\widetilde{x}, \widetilde{a})$. Once
we have rewritten the original word $a_{1}a_{2} \ldots a_{n}x$ in the form
$y \widehat{a}_{1}\ldots \widehat{a}_{n}$, where $y \in X$, it follows that
$(a_{1}\ldots a_{n})_{\mid x} = \widehat{a}_{1} \ldots \widehat{a}_{n}$.

For instance, if $A$ is the automaton from Example \ref{example:moore}, then 
$$ abta \cdot 1 = abt1t = ab0(id)t = a0b(id)t = 0(id)b(id)t.$$
We omit occurrences of the identity state (Definition \ref{def:identity}) from $(a_{1}a_{2} \ldots a_{n})_{\mid x}$. (Note also
that $(a_{1}a_{2} \ldots a_{n})_{\mid x}$ is a string in $(A - \{ id \})^{\ast}$, not an element of a group -- this distinction is
important in Definition \ref{def:planar}.) 
It follows that $abta_{\mid 1} = bt$.
\end{note}

\begin{definition}
Let $G$ be a group defined by an automaton $A$. We let $G_{n}$ denote the \emph{$n$th level stabilizer}, defined as follows:
$$ G_{n} = \{ g \in G \mid g \cdot w = w, \text{ for all } w \in X^{\ast} \text{ such that } |w| \leq n \}.$$
\end{definition}

\begin{definition}
Let $g_{0}, g_{1}, \ldots, g_{n-1}$ be automorphisms of the tree $X^{\ast}$, where $|X| = n$. Assume, without loss of generality,
that $X = \{ 0, 1, \ldots, n-1 \}$. We let $g = (g_{0}, g_{1}, \ldots, g_{n-1})$ be the automorphism defined by the rule
$g(jw) = j g_{j}(w)$, for $j \in X$ and $w \in X^{\ast}$. (Thus, $g$ fixes the top level of the tree and acts like the
automorphism $g_{j}$ on the $j$th branch from the root.)
\end{definition}

Note that if $G$ is a group defined by an automaton $A$ over $X = \{ 0, 1, \ldots, n-1 \}$, and $g \in G_{1}$ can be represented
by a word in $(A- \{ id \})^{\ast}$, then
$g = (g_{\mid 0}, \ldots, g_{\mid n-1})$. We will make extensive use of this fact in subsequent sections.

We will always work with reduced automata.

\begin{definition}
An automaton $A$ is \emph{reduced} if different states $a$ of $A$ induce different functions $a: X^{\ast} \rightarrow X^{\ast}$.
\end{definition}

Any automaton can be reduced, i.e., there is an algorithm which finds a reduced automaton whose states define the same set
of functions $a: X^{\ast} \rightarrow X^{\ast}$ as the given automaton (\cite{N}; pg. 8).

\begin{definition} \label{def:identity}
A state $a$ is called an \emph{identity state} if $a: X^{\ast} \rightarrow X^{\ast}$ is the identity automorphism.
\end{definition}

It is clear from the definition that a reduced automaton can have at most one identity state.

\subsection{Kneading automata of quadratic polynomials} \label{subsection:kneading}

Throughout this section, we assume that $X$ is a two-letter alphabet. The definitions in this section are all drawn from \cite{N}.
In most cases, our definitions look simpler than the ones in \cite{N} because we are restricting our attention to a two-letter
alphabet, while the discussion in \cite{N} is more general. Note in particular that Definition \ref{def:kneadingautomaton}
would be an incorrect definition of kneading automata if we replaced $X$ with a larger set. 

\begin{definition} \label{def:kneadingautomaton} (\cite{N}; pg. 167)
Let $A$ be an invertible reduced automaton over $X = \{ 0, 1 \}$. We say that $A$ is a \emph{kneading automaton}
if
\begin{enumerate}
\item there is only one active state;
\item in the Moore diagram of $A$, each non-identity state has exactly one incoming arrow;
\item at most one outgoing arrow from the active state leads to a non-identity state.
\end{enumerate}
\end{definition}

\begin{definition} \label{def:planar}(\cite{N}; pg. 177)
A kneading automaton $A$ over $X= \{ 0, 1 \}$ is \emph{planar}
if there is some linear ordering $a_{1} \ldots a_{m}$ of
the non-trivial states of $A$ such that
$((a_{1}\ldots a_{m})^{2})_{\mid x}$ is a cyclic shift of 
$a_{1}\ldots a_{m}$ for each $x \in X$.
\end{definition}
 
If $A$ is a kneading automaton over the two-letter alphabet $X$, then there are two general forms that $A$ might take. 
Consider the result of deleting the identity state and all arrows in the Moore diagram for $A$ that lead to the identity state.
It is not too difficult to see that the resulting directed graph $\Gamma_{A}$
is topologically either a circle, or a circle with a
sticker ($[0,1]$) attached at one of its ends. (In the latter case, the active state of $A$ is the unique vertex of degree $1$.)
It is also clearly possible to reconstruct the Moore diagram of $A$ from $\Gamma_{A}$ (since all of the arrows that are missing
from $\Gamma_{A}$ must lead to the identity state).

\begin{definition} (\cite{N}; pg. 183) \label{def:kneadingsequence}
Let $A$ be a kneading automaton. We define the \emph{kneading sequence}
of $A$ as follows.

If $\Gamma_{A}$ is topologically a circle with $m$ edges, then the kneading sequence for $A$ has the form 
$(\ell_{1} \ell_{2} \ldots \ell_{m})^{\omega}$, where $\ell_{1}$ is the label of the (unique) arrow leading from $a_{1}$ (say)
into the active state, $\ell_{2}$ is the label of the arrow leading from $a_{2}$ into $a_{1}$, and so forth, so that
$\ell_{m}$ leads from the active state into $a_{m-1}$. (In other words, we trace the arrows backwards from the active state, 
while recording the labels in the order that they are encountered. When we reach the active state again, having recorded
the string $\ell_{1} \ell_{2} \ldots \ell_{m}$, we define the kneading sequence to be $(\ell_{1} \ldots \ell_{m})^{\omega}$.)

If $\Gamma_{A}$ is topologically a circle with a sticker, then we similarly trace the arrows backward from the active state (which
is necessarily the unique vertex of degree $1$ in $\Gamma_{A}$) and record the labels. The kneading sequence takes the form $u(v)^{\omega}$, 
where $u$ is the (non-empty) label of the sticker, and $v$ is the label of the circle. The latter label $v$ is read from $\Gamma_{A}$
in essentially the same way as before.

We abbreviate the four possible labels $(0,0)$, $(1,1)$, $(0,1)$, and $(1,0)$
by $0$, $1$, $\ast_{0}$, and $\ast_{1}$ (respectively).
\end{definition}

\begin{example}
The automaton $A$ in Example \ref{example:moore} is a kneading automaton. Its associated graph $\Gamma_{A}$ is topologically a
circle with a sticker. If we trace the
arrows backwards from $t$ to $a$, then to $b$, and then to $b$ again,
we read the labels $(1,1)$, $(1,1)$, and $(0,0)$ (respectively).
Since following the arrow backwards from $b$ leads to $b$, the kneading sequence repeats after this. It follows that
the kneading sequence is $11(0)^{\omega}$.
\end{example}

\begin{note}
It is straightforward to check that a kneading automaton can be recovered from its kneading sequence. Note
also that the automaton with the kneading sequence $1(0)^{\omega}$ is different from the automaton with
the kneading sequence $1(00)^{\omega}$ (for example).
\end{note}

\begin{definition} \label{def:preperiodic} (\cite{N}; pg. 184)
We say that a kneading sequence is \emph{pre-periodic}
if it has the form $u(v)^{\omega}$, where $u, v \in X^{\ast}$ are non-trivial strings. (That is, 
if the graph $\Gamma_{A}$ is topologically a circle with a sticker attached.)
\end{definition}

We use a characterization of ``bad isotropy groups" on page 184 of \cite{N} as a definition:

\begin{definition}
A kneading automaton $A$ over $X = \{ 0, 1 \}$ has \emph{bad isotropy groups} if and only if its kneading sequence is pre-periodic
and the word $v$ from Definition \ref{def:preperiodic} is a proper power.
\end{definition}

\begin{definition}
A complex polynomial $p: \mathbb{C} \rightarrow \mathbb{C}$ is called \emph{post-critically finite}
if for each critical point $c$ (i.e., $p'(c) = 0$), the set
$$ \{ p(p(\ldots p(c))) \}$$
is finite.
\end{definition}

\begin{theorem} \label{thm:kneadingsequence}
Let $A$ be an invertible reduced kneading automaton over $X$, where $|X| = 2$. 
If $A$ is  planar and does not have bad isotropy 
groups, then $G(A)$ is the iterated monodromy group 
of a post-critically finite quadratic polynomial $p: \mathbb{C} \rightarrow \mathbb{C}$, and the kneading sequence of $A$ 
is also the kneading sequence of $p$ (up to relabeling of $X$).
\end{theorem}

\begin{proof}
This follows from Nekrashevych's Theorem 6.9.6 \cite{N} and the discussion on page 187 of \cite{N}.
\end{proof}
 
\section{A Sufficient Condition for Subexponential Growth} \label{section:buxperez}

Let $A$ be a kneading automaton with pre-periodic kneading sequence.
It frequently happens that an element $g \in G(A)_{1} = G_{1}$ becomes shorter in total length
when it is written as an ordered pair, i.e., as $g = (g_{\mid 0}, g_{\mid 1})$. 
The arguments of Bux and P\'{e}rez \cite{BP} show that if some positive proportion of group elements $g$ possess this
property, then $G$ will have subexponential growth.

The goal of this section is to give a careful statement and proof of this fact. Our approach is based on \cite{BP}.

\subsection{Length Functions, rewriting rules and weak reduced forms} \label{subsection:length}

We assume, from now on, that $A$ is a kneading automaton with a pre-periodic kneading sequence, and $X = \{ 0, 1 \}$. We let
$t$ denote the (unique) active state of $A$. Our assumptions imply that $t : X^{\ast} \rightarrow X^{\ast}$
is defined by the rule $t(0w) = 1w$; $t(1w) = 0w$. 

We let $S$ be the set of non-trivial, non-active states of $A$. We set $H = \langle S \rangle \leq G(A)$. We will frequently 
write $G$ in place of $G(A)$ when the automaton $A$ is understood.

It is straightforward to check that every generator $a \in A - \{ id \}$ has order $2$ under the current assumptions. We will
assume this fact in what follows. We can therefore represent each group element in $G$ (or $H$) by a word in $(A-\{ id\})^{\ast}$.  

\begin{definition}
Let $\widehat{\ell}: A - \{ id \} \rightarrow \mathbb{R}^{+}$ ( where $\mathbb{R}^{+}$ is the set of positive real numbers).
The assignment $\widehat{\ell}$ determines two functions from $(A-\{ id \})^{\ast}$ to $\mathbb{R}^{+} \cup \{ 0 \}$. For
$w = a_{1}\ldots a_{n}$, ($a_{i} \in A- \{ id \}$),
\begin{enumerate}
\item we set $|w| = \sum_{i=1}^{n} \widehat{\ell}(a_{i})$.
\item we set $\ell(w) = \mathrm{min} \{ |v| \mid v=w \text{ in } G(A) \}$.
\end{enumerate}
\end{definition}

We write $C$ in place of $\ell(t)$. 

\begin{definition} 
Let $s \in S$. We say that $(s_{\mid 0}, s_{\mid 1})$ is the \emph{first-line string production}
of $s$. More generally, if $w \in \left( A - \{ id \} \right)^{\ast}$, we replace each letter $\widehat{s} \in S$ of $w$ 
with the pair $( \widehat{s}_{\mid 0}, \widehat{s}_{\mid 1})$,
and then collect all powers of the active state $t$ at the end of the word using the relation $t(a,b) = (b,a)t$. Finally, we 
multiply all of the pairs together coordinate-wise, without any cancellation or application of relations from $G(A)$ (i.e., 
the multiplication takes place in $(A - \{ id \}^{\ast})$). The resulting expression has the form 
$(\widetilde{w}_{1}, \widetilde{w}_{2})(t)$ for
some words 
$\widetilde{w}_{1}, \widetilde{w}_{2} \in ( A - \{ id \})^{\ast}$; 
it is called the \emph{first-line string production of $w$}. (Note that, here and in what follows, the $t$ in 
$(\widetilde{w}_{1}, \widetilde{w}_{2})(t)$ appears between parentheses because it may be present or not.)

If $w \in (A - \{ id \})^{\ast}$, then we let $(\widetilde{w}_{1}, \widetilde{w}_{2})(t)$ denote the first-line
string production.
\end{definition}

\begin{example}
We consider the automaton $A$ from Example \ref{example:moore}, and find the first-line string production of $w = tabtbabt$.
We find
\begin{align*}
tabtbabt &= t(1,t)(b,a)t(b,a)(1,t)(b,a)t \\
&= (tabb, bata)t, 
\end{align*}
which is the first-line string production of $w$.
\end{example}

\begin{definition} \label{def:T}
For each element $h$ in the group $H = \langle S \rangle$ we fix a word $w_{h} \in S^{\ast}$ such that
\begin{enumerate}
\item $w_{h} = h$ in $G$;
\item If $\widehat{w} \in S^{\ast}$ is equal to $h$ in $G$, then
$|w_{h}| \leq |\widehat{w}|$.
\end{enumerate}
(In other words, $w_{h}$ is a representative for $h \in G$ that uses only letters from $S$, and such
that $|\cdot| : S^{\ast} \rightarrow \mathbb{R}^{+} \cup \{ 0 \}$ is at a minimum over all such representatives.)

We let $T$ denote the set of choices for $w_{h}$, where $h$ ranges over all of $H$.

Let $w = \left( t^{\alpha_{0}} \right) w_{1} t^{\alpha_{1}} w_{2} t^{\alpha_{2}} 
\ldots t^{\alpha_{m-1}} w_{m} \left( t^{\alpha_{m}} \right)$
be an arbitrary word in $(A - \{ id \})^{\ast}$, where each $w_{i} \in S^{\ast}$ and each $\alpha_{i} \geq 1$. 
(The parenthesized terms may be present or not.)
We let $r(w)$ be the result of replacing each subword $w_{i}$ with its representative in $T$. The latter assignment determines a function
$r: (A- \{ id \})^{\ast} \rightarrow (A- \{ id \})^{\ast}$.

A word $w \in (A- \{ id \})^{\ast}$ is in \emph{weak reduced form} if it is in the range of $r$.
\end{definition}

\begin{note} \label{note:nocancellation}
In other words, we put an arbitrary string $w \in (A- \{ id \})^{\ast}$ into weak reduced form by replacing each maximal
string from $S^{\ast}$ by its representative from $T$, without cancelling any $t$s. Thus, it is perfectly acceptable for
a substring of the form $t^{n}$ ($n \geq 2$) to appear in $r(w)$.
\end{note}

\begin{note} \label{note:idempotent}
It is easy to check that $r: (A - \{ id \})^{\ast} \rightarrow (A - \{ id \})^{\ast}$ has 
the property $r(r(u) r(v)) = r(uv)$, where the multiplication takes place in the free semigroup $(A - \{ id \})^{\ast}$.
\end{note}





\begin{definition}
Let $w \in \left( A - \{ id \} \right)^{\ast}$. Let $( \widetilde{w}_{0}, \widetilde{w}_{1})(t)$ be the first-line string production of $w$.
The \emph{first-line production} of $w$ is $(r(\widetilde{w}_{0}), r(\widetilde{w}_{1}))(t)$.

We write $(w_{0}, w_{1})(t)$ for the first-line production of $w$.
\end{definition}

Suppose that $w = (t^{\alpha_{0}})w_{1}t^{\alpha_{1}} w_{2}t^{\alpha_{2}} \ldots w_{m}(t^{\alpha_{m}})$ is in weak reduced form. We set
$$ |w|_{\ast} = \left( C \right) + (m-1)C + \left( C \right)
+ \sum_{i=1}^{m} |w_{m}|,$$
where the parenthesized terms will be absent if the corresponding factors $t^{\alpha_{0}}$, $t^{\alpha_{m}}$
are absent in $w$.

\begin{lemma}
For all $w \in \left( A - \{ id \} \right)^{\ast}$ in weak reduced form, we have $\ell(w) \leq |w|_{\ast} \leq |w|$.
\end{lemma}

\begin{proof}
Let $w = (w_{0})t^{\alpha_{0}}w_{1}t^{\alpha_{1}} \ldots t^{\alpha_{m-1}} w_{m} t^{\alpha_{m}}(w_{m+1})$ be a word
in weak reduced form. (Thus, $w_{i} \in T$, for all $i$.) We assume, for the sake of simplicity, that $w_{0} = w_{m+1} = 1$.
\begin{align*}
|w| &= \sum_{i=0}^{m} \alpha_{i}C + \sum_{i=1}^{m} |w_{i}| \\
&\geq \sum_{i=0}^{m} C + \sum_{i=1}^{m} |w_{i}| \\
&= |w|_{\ast}.
\end{align*}
For each $i \in \{ 0, \ldots, m \}$, let $\beta_{i} \in \{ 0, 1 \}$ be the result of reducing $\alpha_{i}$ modulo $2$.
It follows that 
$t^{\alpha_{0}}w_{1}t^{\alpha_{1}} \ldots t^{\alpha_{m-1}} w_{m} t^{\alpha_{m}} = t^{\beta_{0}}w_{1}t^{\beta_{1}}\ldots t^{\beta_{m-1}}w_{m}t^{\beta_{m}}$
in the group $G(A)$. We have
\begin{align*}
|w|_{\ast} &\geq \sum_{i=0}^{m} \beta_{i}C + \sum_{i=1}^{m} |w_{i}| \\
&= |t^{\beta_{0}} w_{1} t^{\beta_{1}} \ldots t^{\beta_{m-1}} w_{m} t^{\beta_{m}}| \\
&\geq \ell(w),
\end{align*}
where the final inequality follows from the definition of $\ell$.
\end{proof}

\begin{definition} \label{def:admissible}
The length function $\ell: \left( A - \{ id \} \right)^{\ast} \rightarrow \mathbb{R}^{+} \cup \{ 0 \}$ is
called \emph{admissible} if
$|t| + |w| \geq |w_{0}| + |w_{1}|$, for all $w \in T$, where $(w_{0}, w_{1})(t)$ is the first-line production of $w$.
\end{definition}
 
\begin{lemma} \label{lemma:nextline}
Let $w \in \left( A - \{ id \} \right)^{\ast}$ be in weak reduced form, and let $(w_{0}, w_{1})(t)$ be the first-line production of $w$.
\begin{enumerate}
\item $|w|_{\ast} + C \geq |w_{0}| + |w_{1}|$
if $w$ contains fewer blocks of $t$s than blocks of letters from $S$;
\item $|w|_{\ast} \geq |w_{0}| + |w_{1}|$ otherwise.
\end{enumerate}
\end{lemma}

\begin{proof}
We prove (2) first. Let us assume that $w = t^{\alpha_{1}} \widehat{w}_{1} t^{\alpha_{2}} \ldots t^{\alpha_{m}} \widehat{w}_{m}$
is in weak reduced form. We will confine our attention to this subcase, since the subcase
$w = \widehat{w}_{1} t^{\alpha_{1}} \ldots t^{\alpha_{m-1}} \widehat{w}_{m} t^{\alpha_{m}}$
is essentially similar, and the subcase $w = t^{\alpha_{0}} \widehat{w}_{1} t^{\alpha_{1}} \ldots t^{\alpha_{m-1}} \widehat{w}_{m} t^{\alpha_{m}}$
is easier.

We have
\begin{align*}
|w|_{\ast} &= \sum_{i=1}^{m} (|t| + |\widehat{w}_{i}|) \\
&\geq \sum_{i=1}^{m} \left( |(\widehat{w}_{i})_{0}| + |(\widehat{w}_{i})_{1}| \right).
\end{align*}
The desired conclusion $|w|_{\ast} \geq |w_{0}| + |w_{1}|$ now follows from the observation that there is a partition
$\{ P_{0}, P_{1} \}$ of $\{ (\widehat{w}_{i})_{j} \mid  i = 1, \ldots, m; j = 0, 1 \}$ such that each of the
words $w_{k}$ ($k=0,1$) in the first-line production $(w_{0}, w_{1})(t)$ of $w$ is obtained by multiplying the
elements of $P_{k}$ in some order (with each word in $P_{k}$ appearing exactly once), and then applying $r$.   

To prove (1), we note that $w$ must have the form 
$w = \widehat{w}_{1} t^{\alpha_{1}} \ldots t^{\alpha_{m-1}} \widehat{w}_{m} t^{\alpha_{m}} \widehat{w}_{m+1}$. We consider the word $wt$.
By case (2), 
$$ |wt|_{\ast} \geq |(wt)_{0}| + |(wt)_{1}|,$$
but $|wt|_{\ast} = |w|_{\ast} + C$, and $(wt)_{i} = w_{i}$ for $i=0,1$. 
\end{proof}

\subsection{Good and $\epsilon$-good} \label{subsection:good}

\begin{definition}
Let $w \in \left( A - \{ id \} \right)^{\ast}$ 
be in weak reduced form. A subword of $w$ is called \emph{protected} if it begins and ends with $t$, and contains at least one letter 
from $S$.
\end{definition}

\begin{definition}
Write $\widehat{w} \preccurlyeq w$ if 
there is some protected subword $\widetilde{w}$ of $w$ such that
$\widehat{w}$ is a protected subword of some word in the first-line production of $\widetilde{w}$.
\end{definition}

\begin{lemma} \label{lemma:occurrences}
Let $w$ be a word representing an element of $G_{1}$, and let $u_{1}, \ldots, u_{n}$ be protected subwords meeting (at most) in an initial or
terminal block of $t$s. Let $\widehat{u}_{i} \preccurlyeq u_{i}$ for $i = 1, \ldots, n$. The first-line production of $w$ contains occurrences
of each of the $\widehat{u}_{i}$. These occurrences meet (at most) in an initial or terminal block of $t$s and there are at least $n$ distinct 
occurrences in all (one of each).
\end{lemma}

\begin{proof}
Consider the words $\mathring{u}_{1}, \ldots, \mathring{u}_{n}$ obtained from $u_{1}, \ldots, u_{n}$ (respectively) by
omitting the initial and terminal blocks of $t$s. We note that $\mathring{u}_{i}$ and $u_{i}$  have the same
first-line string productions (although they might produce the words in opposite coordinates). We write
$w = v_{1}\mathring{u}_{1} v_{2} \mathring{u}_{2} \ldots \mathring{u}_{n}v_{n+1}$, where:  (i) the $v_{i}$ (for $1 \leq i \leq n$) end with 
blocks of $t$s (and may consist entirely of $t$s), and (ii) the $v_{i}$ (for $2 \leq i \leq n+1$) begin with blocks of $t$s (and may consist
entirely of $t$s), but are not trivial strings in either case. We form the first-line string production 
of $w$. Note that this is done word-by-word in the product 
$w = v_{1}\mathring{u}_{1} v_{2} \mathring{u}_{2} \ldots \mathring{u}_{n}v_{n+1}$. Each word $v_{i}$, $\mathring{u}_{i}$
contributes a string to both the left- and right-hand coordinates (although each may contribute the right-half 
of its first-line string production to the left word in the first-line string production of $w$, or vice versa).

Assume that 
$$ (\widetilde{v}_{1})_{i_{1}} (\widetilde{u}_{1})_{j_{1}} (\widetilde{v}_{2})_{i_{2}} \ldots (\widetilde{u}_{n})_{j_{n}}(\widetilde{v}_{n+1})_{i_{n+1}}$$
is the left-half of the first-line string production of $v_{1}\mathring{u}_{1} \ldots \mathring{u}_{m}v_{m+1}$, where each
string $(\widetilde{v}_{k})_{i_{k}}$, $(\widetilde{u}_{k})_{j_{k}}$ above is either the left or right half of the
first-line string production of $v_{k}$, $\mathring{u}_{k}$ (respectively).
Choose a word $(\widetilde{u}_{k})_{j_{k}}$ ($k \in \{ 1, \ldots, m \}$).
We assume without loss of generality that $\widehat{u}_{k}$ occurs as a protected subword in $(\mathring{u}_{k})_{j_{k}}$,
where $(\mathring{u}_{k})_{j_{k}}$ denotes either the left- or right-half of the first-line production of $\mathring{u}_{k}$.
Since $\widehat{u}_{k} \preccurlyeq \mathring{u}_{k}$, there is some protected subword $u$ such that
$(\widetilde{u}_{k})_{j_{k}} = xuy$ and $r(u) = \widehat{u}_{k}$. 
Thus $r((\widetilde{u}_{k})_{j_{k}}) = r(x)r(u)r(y)$. Since no $t$s are cancelled in an application of $r$, we get that
$r(u) = \widehat{u}_{k}$ occurs as a subword in $w_0$. That is, $\widehat{u}_{k}$ occurs as a protected subword in $w_{0}$. The lemma
now follows.
\end{proof}

\begin{definition} \label{def:reducing} Let $w$ be a word in weak reduced form. A subword $v$ of $w$ is called \emph{reducing} if
$v \equiv t w_{1} t w_{2} t w_{3} \ldots t w_{m}$, where each $w_{i} \in T$, the word $w_{m}$ is followed 
immediately by a $t$ in $w$, $m \geq 1$, and $\ell(v_{0}) + \ell(v_{1}) < |v|_{\ast}$.
\end{definition}

\begin{definition} \label{def:good}
A word $u$ is \emph{good at depth $m$} if there are $u_{1}, \ldots, u_{m}$ such that
$$ u \succcurlyeq u_{1} \succcurlyeq \ldots \succcurlyeq u_{m}, $$
where $u_{m}$ contains a reducing subword $v$.

We set 
$$ C_{u} = \frac{ \ell(v_{0}) + \ell(v_{1})}{|v|_{\ast}}$$
and
$$ \sigma_{u} = |v|_{\ast}.$$
\end{definition}

\begin{definition}
A word $w \in (A - \{ id \})^{\ast}$ is \emph{reduced} if $|w| = \ell(w)$.
\end{definition}

\begin{definition} \label{def:epsilongood1}
Let $0 < \epsilon < 1$. A reduced word $w$ of length $N$ (i.e., $\ell(w) = N$) is \emph{$\epsilon$-good with respect to the good word $u$}
if at least $\epsilon N$ of its length is taken up by occurrences of $u$ which meet in an initial or terminal $t$.
\end{definition}

\begin{definition} 
If $w \in (A - \{ id \})^{\ast}$ represents an element of $G_{k}$ and $\alpha \in \{ 0, 1 \}^{k}$, then we let
$w_{\alpha}$ denote the production of $w$ in position $\alpha$. For instance, $w_{0}$ denotes the left word in the
first-line production of $w$. (Here we consider the ``$0$" branch the left, and the ``$1$" branch the right, half of the
binary tree $X^{\ast}$.) If $k \geq 2$, then $w_{01}$ would denote the right word in the first-line production of $w_{0}$.
We also say that $w_{01}$ is in the \emph{second-line production} of $w$, $w_{011}$ is in the \emph{third-line production}, and so forth.
\end{definition}

\begin{proposition} \label{proposition:bigun}
Let $u$ be a good-at-depth-$m$ word ($m \geq 0$).
There are $0 < \theta < 1$ and $K > 0$ such that, for any $w \in (A- \{ id \})^{\ast}$
satisfying
\begin{enumerate}
\item $|w| = \ell(w)$,
\item $\pi(w) \in G_{m+1}$, and 
\item $w$ is $\epsilon$-good with respect to $u$,
\end{enumerate}
$$ \sum_{\delta} \ell( w_{\delta}) \leq \theta \ell(w) + K.$$
The sum on the left is over all strings $\delta \in \{ 0, 1 \}^{m+1}$.
\end{proposition}

\begin{proof}
Since occurrences of $u$ contribute at least $\epsilon |w|$ to the total length of $w$, there must be at least
$\epsilon(|w|/|u|)$ of them. We let $\beta$ be the number of such occurrences; thus, $ \beta \geq \epsilon(|w|/|u|)$.
We let $v, u_{1}, \ldots, u_{m}$ be the words from Definition \ref{def:good}. 

We choose $\theta$ so that $0 < \theta < 1$ and 
$$ \theta + \frac{\epsilon ( 1 - C_{u}) \sigma_{u}}{|u|} > 1.$$
It is clearly possible to find such $\theta$, since the second term on the left is positive.

We first produce the $m$th line production of $w$. We have either:
\begin{enumerate}
\item $\displaystyle \sum_{\gamma} |w_{\gamma}|_{\ast} \leq \theta |w|$, for all $w$ sufficiently large, or
\item $\displaystyle \sum_{\gamma} |w_{\gamma}|_{\ast} > \theta |w|$, for some sequence of words $w$ such that $|w| \rightarrow \infty$.
\end{enumerate}
The sums on the left are taken over all strings $\gamma \in \{ 0, 1 \}^{m}$.

In the first case, we have
\begin{align*}
\theta |w| + 2^{m}C &\geq  \sum_{\gamma} ( |w_{\gamma}|_{\ast} + C ) \\
&\geq \sum_{\gamma} |w_{\gamma 0}| + |w_{\gamma 1}| \\
&\geq \sum_{\delta} \ell(w_{\delta}),
\end{align*}
where the final sum is over all strings $\delta$ of length $m+1$ in $\{ 0, 1 \}$, and $(w_{\gamma 0}, w_{\gamma 1})$ is the
first-line string production of $w_{\gamma}$.
Thus there is nothing left to prove in this case.

Now assume that we are in the second case. For each $\gamma \in \{ 0, 1 \}^{m}$, we let 
$\beta_{\gamma}$ denote the number of occurrences of $u_{m}$ in $w_{\gamma}$. By Lemma \ref{lemma:occurrences},
we have $\sum_{\gamma} \beta_{\gamma} \geq \beta$.  

The word $w_{\gamma}$ has the form
$$ w_{\gamma} = x_{0}vx_{1}vx_{2}v \ldots x_{\beta_{\gamma}-1}vx_{\beta_{\gamma}},$$
where each $x_{i}$ is a word in weak reduced form, any one of which may be trivial, with the exception of $x_{\beta_{\gamma}}$, 
which must begin with a $t$ if $\beta_{\gamma} > 0$.
We note that each occurrence of $v$ must be followed immediately by a $t$ according to Definition \ref{def:reducing}. It follows that
each $x_{i}$ (for $i>0$) begins with a $t$ if it is not the null string. If some $x_{i}$ ends with a block of $t$s, then
we can combine this block with the word $v$ following it (if any). The altered occurrence of $v$, $v'$, still contributes the same 
strings to the next line string production (though possibly in different places) and even satisfies $|v'|_{\ast} = |v|_{\ast}$, 
so we may ignore the additional powers of $t$ for the sake of the following argument. One can now easily verify that each $x_{i}$
($i>0$) contains at least as blocks of  $t$s as blocks of letters from $S$.

We first claim that 
$$ |w_{\gamma}|_{\ast} + C + (C_{u} - 1)\beta_{\gamma}\sigma_{u} \geq \ell(w_{\gamma 0}) + \ell(w_{\gamma 1}),$$
for each $\gamma \in \{ 0, 1 \}^{m}$.
From the definition of $| \cdot |_{\ast}$ we see that
$$ \ast) \quad |w_{\gamma}|_{\ast} = \sum_{i=0}^{\beta_{\gamma}} |x_{i}|_{\ast} + \sum_{i=1}^{\beta_{\gamma}} |v|_{\ast}.$$
It follows that
\begin{align*}
|w_{\gamma}|_{\ast} + C &= \sum_{i=1}^{\beta_{\gamma}} |x_{i}|_{\ast} + \left( |x_{0}|_{\ast} + C \right) + \beta_{\gamma}\sigma_{u} \\
&\geq \sum_{i=0}^{\beta_{\gamma}} \left[ \ell( (x_{i})_{0}) + \ell( (x_{i})_{1}) \right] + \beta_{\gamma}C_{u}\sigma_{u} \\
&\geq \ell(w_{\gamma 0}) + \ell(w_{\gamma 1}).
\end{align*}
Here the first equality follows directly from $\ast)$. The expression on the second line is the $|\cdot|$-length
of the strings that result from taking the first-line productions of the words $x_{i}$ and $v$, and reducing the results (i.e., substituting 
shortest-length strings) individually. 
The first inequality follows from Lemma \ref{lemma:nextline}. The final inequality is now immediate. We also have
\begin{align*}
|w_{\gamma}|_{\ast} + C + \beta_{\gamma}C_{u}\sigma_{u} - \beta_{\gamma}\sigma_{u} &= \sum_{i=0}^{\beta_{\gamma}} |x_{i}|_{\ast} + C + \beta_{\gamma}C_{u}\sigma_{u} \\
&\geq \sum_{i=0}^{\beta_{\gamma}} \left[ \ell((x_{i})_{0}) + \ell((x_{i})_{1}) \right] + \beta_{\gamma}C_{u}\sigma_{u}.
\end{align*}
The claim follows readily.

Applying the claim (and Lemma \ref{lemma:nextline}, repeatedly) we get
\begin{align*}
|w| + (1 + 2 + 2^{2} + \ldots + 2^{m})C &\geq \sum_{\gamma} \left( |w|_{\ast} + C \right) \\
&\geq \sum_{\gamma} \left( |w|_{\ast} + C \right) + \beta \sigma_{u} (C_{u} - 1) \\
&\geq \sum_{\delta} \ell( w_{\delta}).
\end{align*}
At worst,
$$\sum_{\gamma} \left( |w|_{\ast} + C \right) + \beta \sigma_{u} (C_{u} - 1) > \theta |w|$$
for infinitely many $w$ such that $|w| \rightarrow \infty$ (otherwise there is nothing left to prove).

We conclude that 
\begin{align*}
|w| + (1 + 2 + 2^{2} + \ldots + 2^{m})C &\geq \theta |w| + (1 - C_{u})\beta \sigma_{u} \\
&\geq \theta |w| + (1-C_{u}) \epsilon \frac{|w|}{|u|} \sigma_{u} \\
&= |w| \left( \theta + \frac{\sigma_{u}(1-C_{u})\epsilon}{|u|} \right).
\end{align*}
for all such $w$. It follows that
$$ |w| + (1 + 2 + \ldots + 2^{m})C \geq D|w|,$$
where $C$, $D$, and $m$ are constants, $D>1$, and $|w| \rightarrow \infty$. This is a contradiction.
\end{proof}

\begin{definition} \label{def:epsilongood2}
Let $0 < \epsilon < 1$. Let $\mathcal{U} = \{ u_{1}, \ldots, u_{n} \}$ be such that $u_{i}$
is a good-at-depth-$m_{i}$ word, possibly for varying $m_{i}$. A reduced 
word $w$ of length $N$ is \emph{$\epsilon$-good with respect to $\mathcal{U}$}
if at least $\epsilon N$ of its length is taken up by occurrences of words from $\mathcal{U}$ which meet (at most) in initial or terminal 
blocks of $t$s.
\end{definition}

\begin{theorem}\label{theorem:final}
Let $\mathcal{U} = \{ u_{1}, \ldots, u_{n} \}$ be as in Definition \ref{def:epsilongood2}. Let $M = \mathrm{max}\{ m_{i} \}$.
There are $0 < \theta < 1$ and $K > 0$ such that, for all reduced words $w \in (A - \{ id \})^{\ast}$ satisfying
\begin{enumerate}
\item $|w| = \ell(w)$,
\item $\pi(w) \in G_{M+1}$, and
\item $w$ is $\epsilon$-good with respect to $\mathcal{U}$,
\end{enumerate}
$$ \sum_{\delta} \ell(w_{\delta}) \leq \theta \ell(w) + K.$$
The sum on the left is over all strings $\delta \in \{ 0, 1 \}^{M+1}$.
\end{theorem}

\begin{proof}
Let $w$ be an arbitrary word satisfying the given conditions. It follows that $w$ is $\epsilon / n$-good with respect to some word $u_{i}$, which is good
at depth $m_{i}$. It follows from Proposition \ref{proposition:bigun} that
$$ \sum_{\delta_{i}} \ell( w_{\delta_{i}}) \leq \theta(u_{i}, \epsilon/n) \ell(w) + K(u_{i}, \epsilon/n),$$
where the sum on the left is over all $\delta_{i} \in \{ 0, 1 \}^{m_{i}+1}$. 

Lemma \ref{lemma:nextline} implies that
$$ \sum_{\delta} \ell(w_{\delta}) \leq \sum_{\delta_{i}} \ell(w_{\delta_{i}}) + \left( 2^{m_{i}+1} + 2^{m_{i} + 2} + \ldots + 2^{M} \right) C.$$
We set $K_{2}(u_{i}, \epsilon/n) = (2^{m_{i}+1} + \ldots + 2^{M})C$. (In fact, $K_{2}$ depends only on $u_{i}$.)

We have
$$ \sum_{\delta} \ell(w_{\delta}) \leq \theta(u_{i}, \epsilon/n)\ell(w) + K(u_{i}, \epsilon/n) + K_{2}(u_{i}, \epsilon/n).$$
We set $\theta(\mathcal{U}, \epsilon) = \mathrm{max} \{ \theta(u_{i}, \epsilon/n) \}$,
 $K(\mathcal{U}, \epsilon) = \mathrm{max} \{ K(u_{i}, \epsilon/n) \}$, and
$K_{2}(\mathcal{U}, \epsilon) = \mathrm{max} \{ K_{2}(u_{i}, \epsilon/n) \}$. 

It now follows easily that
$$ \sum_{\delta} \ell(w_{\delta}) \leq \theta(\mathcal{U}, \epsilon)\ell(w) + K(\mathcal{U}, \epsilon) + K_{2}(\mathcal{U}, \epsilon)$$
for all words $w$ satisfying the hypotheses.
\end{proof}

\subsection{Subexponential Growth} \label{subsection:growth}

\begin{definition}
Let $\mathcal{U}$ be a collection of good words. A reduced word $w$ is \emph{$\mathcal{U}$-bad}
if no word of $\mathcal{U}$ occurs as a subword of $w$.
\end{definition}

\begin{theorem} \label{thm:intgrowth}
For $i=1, \ldots, n$, let $u_{i}$ be a good-at-depth
$m_{i}$  word. Let $\mathcal{U} = \{ u_{1}, \ldots, u_{n} \}$. If there
is some $M>0$ such that, for all $L>0$, there are at most $M$ $\mathcal{U}$-bad words $w$ of length $L$ (i.e., $\ell(w) = L$), then
$G(A)$ has subexponential growth.
\end{theorem}

\begin{proof}
For each $r>0$ and all small $\epsilon > 0$, we estimate the number $b_{\epsilon}(r)$ of $\epsilon$-bad words $w \in G_{N+1}$ of 
length $\ell$  precisely $r$. We will make the (harmless) assumption that each word has integral length.
We have the following estimate:
$$ b_{\epsilon}(r) \leq \binom{r+1}{\lfloor \epsilon r \rfloor + 1} M^{1 + \lfloor \epsilon r \rfloor} |\mathcal{U}|^{\lfloor \epsilon r \rfloor}.$$
The right half of the inequality assumes that $\lfloor \epsilon r \rfloor$ good words appear in $w$ (an absolute worst case). The 
factor $|\mathcal{U}|^{\lfloor \epsilon r \rfloor}$ counts the possible selections of those good words. The $\lfloor \epsilon r \rfloor$ good
words divide $w$ into $\lfloor \epsilon r \rfloor + 1$ pieces, all of which must be bad words. 
The binomial coefficient counts the number of solutions in non-negative integers 
to the inequality
$$\ell_{1} + \ldots + \ell_{\lfloor \epsilon r \rfloor + 1} \leq r - \lfloor \epsilon r \rfloor,$$ 
which over-counts the number of possible choices for the lengths of these bad words. Since the number of bad words of a given length
is uniformly bounded by $M$, we find that there are at most $M^{1+\lfloor \epsilon r \rfloor}$ choices for these words.

Next, we consider the number $B_{\epsilon}(n)$ of reduced words $w \in G_{N+1}$ of length less than or equal to $n$.
\begin{align*}
B_{\epsilon}(n) &= \sum_{r=0}^{n} b_{\epsilon}(r) \\
&\leq \sum_{r=0}^{n} \binom{r+1}{\lfloor \epsilon r \rfloor + 1} M^{1+\lfloor \epsilon r \rfloor} |\mathcal{U}|^{\lfloor \epsilon r \rfloor} \\
&\leq M^{1+ \epsilon n}|\mathcal{U}|^{\epsilon n} \sum_{r=0}^{n} \binom{r+1}{\lfloor \epsilon r \rfloor + 1} \\
&\leq M^{1+ \epsilon n}|\mathcal{U}|^{\epsilon n} (n+1) \binom{n+1}{\lfloor \epsilon n \rfloor + 1} \\
&\leq M^{1+ \epsilon n}|\mathcal{U}|^{\epsilon n} (n+1) \frac{(n+1)^{\lfloor \epsilon n \rfloor + 1}}{(\lfloor \epsilon n \rfloor + 1)!}.
\end{align*}
For large $n$, the latter quantity is approximately
$$ M^{1 + \epsilon n}|\mathcal{U}|^{\epsilon n} (n+1) e^{1+\lfloor \epsilon n \rfloor} 
\frac{1}{\sqrt{2 \pi ( \lfloor \epsilon n \rfloor + 1)}} \left( \frac{n+1}{\lfloor \epsilon n \rfloor + 1} \right)^{\lfloor \epsilon n \rfloor + 1},$$
by an application of Stirling's Formula to the quantity $(\lfloor \epsilon n \rfloor + 1)!$.

Now we suppose (for a contradiction) that $G$ (thus, $G_{N+1}$) has exponential growth. Thus, there is some $\lambda > 1$ such that the ball
of radius $n$ in $G_{N+1}$ has roughly $\lambda^{n}$ elements. The estimate of $B_{\epsilon}(n)$ implies that we can choose $\epsilon$ sufficiently
small that $B_{\epsilon}(n) \leq (\lambda_{1})^{n}$, for some $1 < \lambda_{1} < \lambda$. Thus, the proportion 
of $\epsilon$-good elements in the ball of radius $n$ to the total is at the least (roughly) $1 - \lambda_{1}/\lambda$. It follows that there
is some positive proportion of reduced words $w \in G_{N+1}$ that satisfy the inequality in Theorem \ref{theorem:final}. In view of Proposition 10 from
\cite{BP}, we are done. 
\end{proof}

\section{Examples of groups with intermediate growth} \label{section:main}

In practice, we will want to have one more type of next-line production, which is intermediate between the first-line string
production and the first-line production. 

\begin{definition} \label{def:special}
Let $v = tw_{1}tw_{2}t \ldots tw_{m}$, where each $w_{i} \in T$. Consider the following operation:
Replace each $w_{i}$ with its first-line production $((w_{i})_{0}, (w_{i})_{1})$
and collect all powers of $t$ at the end of the word, using the relation
$t(a,b) = (b,a)t$. Multiply the pairs coordinate-by-coordinate
with no cancellation (i.e., the multiplication occurs in $(A-\{ id \})^{\ast}$). The 
result is called the \emph{special production} of $v$.
\end{definition}

\begin{lemma} \label{lemma:special}
Let $v$ be as in Definition \ref{def:special}, and suppose that $(\widehat{v}_{0}, \widehat{v}_{1})(t)$ is
its special production. If $|\widehat{v}_{i}| > \ell(\widehat{v}_{i})$ for $i = 0$ or $1$, then 
$|v|_{\ast} > \ell(v_{0}) + \ell(v_{1})$.
\end{lemma}

\begin{proof}
We note that the inequality $|v|_{\ast} \geq |\widehat{v}_{0}| + |\widehat{v}_{1}|$ follows immediately from
Lemma \ref{lemma:nextline}, from which the desired conclusion follows by Note \ref{note:idempotent}.
\end{proof}

The groups associated to the automata with the kneading sequences $1(0^{k})^{\omega}$, $11(0)^{\omega}$, 
and $0(011)^{\omega}$ will be proved to have subexponential growth in this section. In fact, all also have superpolynomial
growth, since each group $G$ is commensurable with $G \times G$, and this is condition is known to imply superpolynomial growth (see \cite{GP},
for instance).

\subsection{The Case of $1(0^{k})^{\omega}$} \label{subsection:first}

We consider the automata $A_{k}$ ($k \geq 2$) with kneading sequence $1(0^{k})^{\omega}$.

The graph $\Gamma_{A_{4}}$ is depicted in Figure \ref{figure:2}, which also indicates our convention for naming the states 
of $A_{k}$ ($k \geq 2$). Namely, $x_{0}$ is the state adjacent to $t$ (the unique active state), $x_{1}$ is the first 
state we encounter while tracing directed edges backwards from $x_{0}$, $x_{2}$ is the second such state, and so forth.

\begin{figure}[!h] 
\begin{center}
\includegraphics{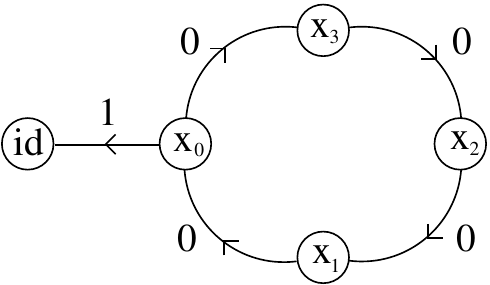}
\end{center}
\caption{The graph $\Gamma_{A}$ for the automaton with the kneading sequence $1(0^{4})^{\omega}$.}
\label{figure:2}
\end{figure}

\begin{lemma} \label{lemma:10k}
Let $x_{0}$, $x_{1}$, $\ldots$, $x_{k-1}$ be the inactive states in the automaton $A_{k}$.
$$ \langle x_{0}, \ldots, x_{k-1} \rangle \cong \left( \mathbb{Z}/2\mathbb{Z} \right)^{k}.$$
\end{lemma}

\begin{proof}
The action of $x_{i}$ on a string $n_{1}n_{2} \ldots n_{m} \in X^{\ast}$
can be described as follows. If $n_{1}n_{2} \ldots n_{m}$ begins with a string of exactly $j$ $0$'s ($j \geq 0$)
followed by a $1$, then $x_{i} \cdot n_{1} \ldots n_{m} = n_{1} \ldots n_{m}$ if $j \not \equiv i$ modulo $k$.
If $j \equiv i$ modulo $k$, then $x_{i} \cdot n_{1}n_{2} \ldots n_{j}n_{j+1}\widehat{n}_{j+2}\ldots$
where $\widehat{n}_{j+2}$ is $0$ if $n_{j+2}$ is $1$, and $1$ if $n_{j+2}$ is $0$. (If $n_{1}n_{2} \ldots n_{m}$
contains no $1$, or if $m = j+1$, then $x_{i} \cdot n_{1} \ldots n_{m} = n_{1} \ldots n_{m}$.)
It follows easily from this description that each $x_{i}$ has order $2$, and that any two elements of
$\{ x_{0}, \ldots, x_{k-1} \}$ commute (since they have disjoint supports). 

We claim that the words $x_{i_{1}}\ldots x_{i_{\alpha}}$ ($\alpha \geq 0$) are all distinct, where
$i_{1}$, $\ldots$, $i_{\alpha}$ is an increasing sequence and $\{ i_{1}, \ldots, i_{\alpha} \} \subseteq \{ 0, \ldots, k-1 \}$.
In fact
$w := x_{i_{1}}x_{i_{2}} \ldots x_{i_{\alpha}}$ has a description analogous to that of $x_{i}$: 
If $n_{1}n_{2} \ldots n_{m}$ begins with a string of exactly $j$ $0$'s ($j \geq 0$)
followed by a $1$, then $w \cdot n_{1} \ldots n_{m} = n_{1} \ldots n_{m}$ if $j \not \equiv i$ modulo $k$, for any 
$i \in \{ i_{1}, \ldots, i_{\alpha} \}$.
If $j \equiv i$ modulo $k$ for some such $i$, then $w \cdot n_{1}n_{2} \ldots n_{j}n_{j+1}\widehat{n}_{j+2}\ldots$
where $\widehat{n}_{j+2}$ is $0$ if $n_{j+2}$ is $1$, and $1$ if $n_{j+2}$ is $0$. (If $n_{1}n_{2} \ldots n_{m}$
contains no $1$, or if $m = j+1$, then $w \cdot n_{1} \ldots n_{m} = n_{1} \ldots n_{m}$.)
It follows immediately that all such $x_{i_{1}}\ldots x_{i_{\alpha}}$ are distinct, for distinct sequences
$i_{1}$, $\ldots$, $i_{\alpha}$. 

Thus, $\langle x_{0}, \ldots, x_{k-1} \rangle$ contains at least $2^{k}$ elements,
which implies that $\langle x_{0}, \ldots, x_{k-1} \rangle \cong \left( \mathbb{Z}/2\mathbb{Z} \right)^{k}$, since our earlier
argument shows that the former group is a quotient of the latter.
\end{proof}

We assign weights to the states $t$, $x_{0}$, $\ldots$, $x_{k-1}$ as follows:
\begin{align*}
\widehat{\ell}(t) &= (k+2)^{2} \\
\widehat{\ell}(x_{i}) &= (k+1-i) \quad (0 \leq i \leq k-1)
\end{align*}
We let $T = \{ x_{i_{1}}\ldots x_{i_{\alpha}} \mid \alpha \geq 0; \, \, 0 \leq i_{1} < i_{2} < \ldots < i_{\alpha} \leq k-1 \}$.
It is not difficult to see that this choice $T$ has the property required by Definition \ref{def:T}.

\begin{lemma} \label{lemma:10kadmissible}
The length function $\ell: ( A - \{ id \})^{\ast} \rightarrow \mathbb{R}^{+} \cup \{ 0 \}$ induced by $\widehat{\ell}$
is admissible. In fact, $|t| + |w| > |w_{0}| + |w_{1}|$
for all $w \in T - \{ x_{0}x_{1}x_{2} \ldots x_{k-1} \}$ (i.e., for all words in $T$ except for the unique one using all $k$ states
$x_{0}$, $\ldots$, $x_{k-1}$).
\end{lemma}

\begin{proof}
Note that
\begin{align*}
x_{0} &= (x_{k-1}, t); \\
x_{i} &= (x_{i-1}, 1) \quad (1 \leq i \leq k-1). 
\end{align*}
We first consider the case in which $w$ contains no occurrence of $x_{0}$ (i.e., $i_{1} > 0$). In this case
$|w_{1}| = 0$ and $w_{0} \in T$. It follows 
that
$$ |w_{0}| \leq \sum_{i=0}^{k-1} |x_{i}| = \frac{k^{2} + 3k}{2} < |t|,$$
from which the strict inequality $|t| + |w| > |w_{0}| + |w_{1}|$
follows immediately.

Next, we suppose that $w$ contains an occurrence of $x_{0}$ (i.e., $i_{1} = 0$), but $w \neq x_{0}x_{1} \ldots x_{k-1}$.
In this case $|w_{1}| = |t|$. Thus, to establish $|t| + |w| > |w_{0}| + |w_{1}|$, we want to show that $|w| > |w_{0}|$.
We write $w = x_{0}x_{i_{2}}\ldots x_{i_{\alpha}}$. Note that $\alpha < k$.
\begin{align*}
|w_{0}| &= |x_{k-1}x_{i_{2}-1}x_{i_{3}-1}\ldots x_{i_{\alpha}-1}| \\
&= 2 + \sum_{\beta = 2}^{\alpha}  |x_{i_{\beta}-1}| \\
&= 2+ \sum_{\beta = 2}^{\alpha} (k+1 - (i_{\beta}-1)) \\
&= (\alpha +1) + \sum_{\beta = 2}^{\alpha} (k+1 - i_{\beta}) \\
&< (k+1) + \sum_{\beta = 2}^{\alpha} (k+1 - i_{\beta}) \\
&= |w|, 
\end{align*}
as required.

Finally, one easily sees that $|t| + |w| = |w_{0}| + |w_{1}|$ if $w = x_{0}x_{1}\ldots x_{k-1}$.
\end{proof}

\begin{theorem}
Each $G(A_{k})$ ($k \geq 2$) has subexponential growth.
\end{theorem}

\begin{proof}
We first nominate a set $\mathcal{U}$ of good words.
Let $\mathcal{U} = \{ twt \mid w \in T - \{ x_{0}x_{1}\ldots x_{k-1} \} \}$.
Each word in this collection is good at depth $0$ (the subword $tw$ is a reducing word in each case, by 
Lemma \ref{lemma:10kadmissible}).

By Theorem \ref{thm:intgrowth}, to prove that the growth is subexponential
it is enough to show that there is an $M>0$ such that, for all $L > 0$, there are at most $M$ $\mathcal{U}$-bad 
words of length $L$. Let us consider a $\mathcal{U}$-bad word of the form $tw_{1}tw_{2}t\ldots tw_{m}t$
($m>0$). It is clear that each $w_{i}$ must be $x_{0}x_{1} \ldots x_{k-1}$, so there is exactly one 
$\mathcal{U}$-bad word of this form for any $m$. A general bad word has the
form $(w_{0})tw_{1}t\ldots tw_{m}t(w_{m+1})$, which shows that the number of such words is bounded above by
a uniform constant that is independent of $m$. This easily implies the existence of the required $M$. 
\end{proof}

\begin{note}
All of the groups in this class have bad isotropy groups, so Theorem \ref{thm:kneadingsequence}
does not guarantee that these groups are the iterated monodromy groups of complex polynomials.
\end{note}

\subsection{The Case of $11(0)^{\omega}$} \label{subsection:second}

We now consider the group $G(A)$ of the kneading automaton $A$ with kneading sequence $11(0)^{\omega}$. This 
automaton has already appeared as Example \ref{example:moore}. The group $G(A)$ is generated by the automorphisms
$t$, $a = (1,t)$, and $b = (b,a)$.

\begin{lemma}
The group $\langle a, b \rangle$ is isomorphic to $D_{4}$, the dihedral group of order $8$.
\end{lemma}

\begin{proof}
We first note that $a^{2} = (1, t^{2}) = (1,1) = 1$. We also have $b^{2} = (b^{2}, a^{2}) = (b^{2}, 1)$. It follows easily by
induction on the length $m$ of a word $n_{1}\ldots n_{m} \in \{0, 1 \}^{\ast}$ that $b^{2}$ also acts as the identity. Thus, $b^{2} = 1$.
$$ (ab)^{4} = (b,at)^{4} = (1, (at)^{4}).$$
Also,
$$ (at)^{4} = (atat)^{2} = \left[ (1,t)t(1,t)t \right]^{2} = (t,t)^{2} = 1.$$
It follows that $(ab)^{4} = 1$.

We've now shown that $\langle a, b \rangle$ is a quotient of $D_{4}$. We define a homomorphism
$\phi: \langle a, b \rangle \rightarrow \langle a, t \rangle$ by $\phi(a) = t$; $\phi(b) = a$.
(This homomorphism is restriction to the second coordinate, or restriction to the right branch of the
tree $\{ 0, 1 \}^{\ast}$.)
Now one notes that $\langle a, t \rangle$ has order $8$ as follows. We consider $00$, $01$, $10$, $11 \in \{ 0, 1 \}^{\ast}$.
Relabel these vertices $1$, $2$, $3$, and $4$, respectively. It is straightforward to check that $a$ acts as the permutation
$(34)$ and $t$ acts as $(13)(24)$. It follows that $|\langle a, t \rangle| \geq 8$, so
$|\langle a, b \rangle| \geq 8$. Thus, $|\langle a, b \rangle| = 8$,
so $\langle a, b \rangle \cong D_{4}$.
\end{proof}

We define $\widehat{\ell}: \{ a, b, t \} \rightarrow \mathbb{R}^{+}$ by the rule
$\widehat{\ell}(a) = \widehat{\ell}(b) = \widehat{\ell}(t) = 1$. We now fix representatives $w_{h} \in S^{\ast}$
for each $h \in \langle a, b \rangle$ as in Definition \ref{def:T}. Set
$$ T = \{ 1, a, ab, aba, abab, b, ba, bab \}.$$
The first-line productions of the elements of $T$ are (respectively) as follows:
$$ (1,1), (1,t), (b,ta), (b,tat), (1,tata), (b,a), (b,at), (1,ata).$$
One can easily check that the length function  $\ell: (A - \{ id \})^{\ast} \rightarrow \mathbb{R}^{+} \cup \{ 0 \}$
associated to $\widehat{\ell}$ is admissible. We summarize the relevant calculations in a table.

\begin{figure} [!h]
\begin{center}
\begin{tabular}{|c|c|c|} \hline
$w$ & $|t| + |w|$ & $|w_{0}| + |w_{1}|$ \\ \hline \hline
$1$ & $1$ & $0$ \\ \hline
$a$ & $2$ & $1$ \\ \hline
$ab$ & $3$ & $3$ \\ \hline
$aba$ & $4$ & $4$ \\ \hline
$abab$ & $5$ & $4$ \\ \hline
$b$ & $2$ & $2$ \\ \hline
$ba$ & $3$ & $3$ \\ \hline
$bab$ & $4$ & $3$ \\ \hline
\end{tabular}
\end{center}
\end{figure}

We let $\alpha_{n}$ be the element of $T$ of word-length $n$ beginning with a. Let $\beta_{n}$ denote the element of
word-length $n$ beginning with $b$. Thus, $T = \{ 1, \alpha_{1}, \alpha_{2}, \alpha_{3}, \alpha_{4}, \beta_{1}, \beta_{2}, \beta_{3} \}$.

\begin{proposition} \label{prop:good110}
The following families $Pi$ of 
words in the generators $\{a, b, t \}$ are good, where each box $\Box$ represents an occurrence
of a string from $\{ \beta_{1}, \beta_{2}, \alpha_{2} \}$:
\begin{enumerate}
\item[P0.] $\, \, t \alpha_{1} t$, $t \alpha_{4} t$, $t \beta_{3} t$, $t \alpha_{3} t$;
\item[P1.] $\, \, t\Box t \beta_{1} t \Box t \beta_{1} t \Box t$;
\item[P2.] $\, \, t \beta_{2} t \Box t \beta_{1} t \Box t \beta_{2} t$;
\item[P3.] $\, \, t \alpha_{2} t \Box t \beta_{1} t \Box t \beta_{2} t$;
\item[P4.] $\, \, t \beta_{2} t \Box t \beta_{1} t \Box t \alpha_{2} t$;
\item[P5.] $\, \, t \alpha_{2} t \Box t \beta_{1} t \Box t \alpha_{2} t$;
\item[P6.] 
$\, \, t \beta_{2} t \Box t \beta_{2} t \Box t \beta_{2} t \Box t \alpha_{2} t \Box t \alpha_{2} t \Box t \alpha_{2} t$, and
\item[P7.] $\, \, t \alpha_{2} t \Box t \beta_{2} t$.
\end{enumerate}
Note that the union $\mathcal{U} = \bigcup_{i=0}^{7} P_{i}$ is also finite.
\end{proposition}

\begin{proof}
We consider each of the above families of words.
\begin{enumerate}
\item[P0.] It is clear from the table that $t \alpha_{1}$, $t \alpha_{4}$, and $t \beta_{3}$ are reducing words in 
$t \alpha_{1}t$, $t \alpha_{4}t$, and $t \beta_{3}t$ (respectively), so the latter words are good at depth $0$.
The first-line production of $t \alpha_{3} t$ is $(tat, b)$. It follows that $tat = t\alpha_{1} t \preccurlyeq t \alpha_{3}t$, so 
$t \alpha_{3} t$ is good at depth $1$.
\item[P1.] The special production of $t \Box t \beta_{1} t \Box t \beta_{1} t \Box$ has the form
$( \underline{\hspace{8pt}}, babab)t$. Since $r(babab) = aba$ and $|aba| < |babab|$, each such word is a reducing word in 
$t \Box t \beta_{1} t \Box t \beta_{1} t \Box t$ by Lemma \ref{lemma:special}. It follows that $t \Box t \beta_{1} t \Box t \beta_{1} t \Box t$
is good at depth $0$. 
\item[P2.] The first-line production of a word $w$ in this family takes the form $(at \alpha_{4}t, \underline{\hspace{8pt}} )$. It follows that
$t \alpha_{4} t \preccurlyeq w$, for all words $w$ in $P2$. Therefore, each such $w$ is good at depth $1$, since
$t \alpha_{4} t \in P0$ is good at depth $0$.
\item[P3.] The special production of a word in this family takes the form $(tababat, \underline{\hspace{8pt}} )$. 
It follows (as in P1) that
each is good at depth $0$; the reducing word omits the final $t$.
\item[P4.] The first-line production of a word $w \in P4$ has the form $(atbabta, \underline{\hspace{8pt}} )$. Thus,
$t \beta_{3} t \preccurlyeq w$ for all words $w$ in $P4$. Since $t \beta_{3} t \in P0$ is good at depth $0$, 
we conclude that each $w \in P4$ is good at depth $1$.
\item[P5.] The first-line production of a word $w \in P5$ has the form $(tababta, \underline{\hspace{8pt}} )$. It follows that
$t \alpha_{4} t \preccurlyeq w$ for all words $w \in P5$. It now follows that each $w$ is good at depth $1$, since $t \alpha_{4} t$ is 
good at depth $0$.
\item[P6.] The first-line production of a word $w$ in $P6$ has the form $(at\beta_{2}t\beta_{2}t\beta_{1}t\alpha_{2}t\alpha_{2}ta, \underline{\hspace{8pt}})$. 
It follows that $t \beta_{2} t \beta_{2} t \beta_{1} t \alpha_{2} t \alpha_{2} t \preccurlyeq w$. The former
word is in family $P4$, so it is good at depth $1$. Thus $w$ is good at depth $2$.
\item[P7.] The first-line production of a word $w$ in $P7$ has the form $(t \alpha_{3} t, \underline{\hspace{8pt}})$. Thus
we have $t \alpha_{3} t \preccurlyeq w$, so $w$ is good at depth $2$ because $t \alpha_{3} t \in P0$ is good at depth $1$.
\end{enumerate}
The final statement is clear.
\end{proof} 

\begin{theorem} \label{thm:110growth}
The group $G(A)$ of the automaton $A$ with kneading sequence $11(0)^{\omega}$ has subexponential growth.
The group $G(A)$ is also the iterated monodromy group of a complex post-critically finite quadratic polynomial. 
\end{theorem}

\begin{proof}
We prove the second statement first. In fact, by Theorem \ref{thm:kneadingsequence}, it is enough to show that
$A$ is planar. This follows easily from the observation that
$abtabt_{\mid 0} = bta$ and $abtabt_{\mid 1} = tab$.
 
We now turn to the first statement. 
Consider the collection of all reduced words that are $\mathcal{U}$-bad (where $\mathcal{U}$ is as in Proposition \ref{prop:good110}).
By Theorem \ref{thm:intgrowth}, it is enough to show that there is $M>0$ such that, for a given $L$, the number of $\mathcal{U}$-bad words
of length $L$ is less than $M$. It is clear that we may restrict our attention to large $L$.

We first consider bad words $w$ that begin and end with $t$. Thus $w = t w_{1} t w_{2} t \ldots t w_{m} t$ (for some large integer
$m$). It is clear (from the description of $P0$) that $w_{i} \in \{ \beta_{1}, \beta_{2}, \alpha_{2} \}$ for each $i$. We consider the possibilities
for the sequence $w_{1}, w_{3}, w_{5}, \ldots, w_{2k-1}$, where $2k-1$ is the largest odd number less than or equal to $m$. 

First, suppose that $w_{1} = \alpha_{2}$. It follows that $w_{1}, w_{3}, w_{5}, \ldots, w_{2k-1}$ has one of the forms:
\begin{enumerate}
\item $\alpha_{2}, \alpha_{2}, \ldots, \alpha_{2}, \alpha_{2}$;
\item $\alpha_{2}, \alpha_{2}, \ldots, \alpha_{2}, \beta_{1}$;
\item $\alpha_{2}, \alpha_{2}, \ldots, \alpha_{2}, \beta_{1}, \beta_{1}$ (if $2k-1< m$).
\end{enumerate}
Indeed, $\beta_{2}$ cannot follow $\alpha_{2}$, since that would create a subword from $P7$. If $\beta_{1}$ follows $\alpha_{2}$
except in the last place (or in the second-to-last place, if $2k-1 < m$), then the next word in the sequence $w_{1}, \ldots, w_{2k-1}$ 
will be $\beta_{1}$, $\beta_{2}$, or $\alpha_{2}$, which will create a subword from $P1$, $P3$, or $P5$ (respectively). 

Now we consider the sequences $w_{1}, \ldots, w_{2k-1}$ such that $w_{1} = \beta_{2}$ and some subsequent $w_{i}$ is $\alpha_{2}$
(for an odd subscript $i$). We claim that no occurrence of $\beta_{1}$ can appear between $w_{1} (= \beta_{2})$ and the earliest
occurrence of $\alpha_{2}$. If $w_{j} = \beta_{1}$ is the earliest such occurrence, then $w_{j-2} = \beta_{2}$ and any choice
of $w_{j+2} \in \{ \beta_{1}, \beta_{2}, \alpha_{2} \}$ yields a subword from $P1$, $P2$, or $P4$ (respectively). This proves the claim.
Now note that, once an $\alpha_{2}$ occurs in $w_{1}, w_{3}, \ldots, w_{2k-1}$, the remainder of the sequence takes one of the forms
enumerated above, from the case in which $w_{1} = \alpha_{2}$. In view of $P6$, the only possibilities are that
$w_{1}, w_{3}, \ldots, w_{2k-1}$ begins with $2$ or fewer occurrences of $\beta_{2}$ followed by a sequence of one of the forms (1)-(3),
or that $w_{1}, \ldots, w_{2k-1}$ begins with a long string of $\beta_{2}$'s, followed by a sequence of one of the forms (1)-(3) that 
contains $2$ or fewer occurrences of $\alpha_{2}$.

Now we consider the sequences $w_{1}, \ldots, w_{2k-1}$ such that $w_{1} = \beta_{2}$ and $w_{1}, \ldots, w_{2k-1} \in \{ \beta_{1}, \beta_{2} \}$.
In view of $P1$, $w_{1}, \ldots, w_{2k-1}$ contains neither a subsequence of the form $\beta_{1}, \beta_{1}, \beta_{2}$,
nor one of the form $\beta_{1}, \beta_{1}, \beta_{1}$. In view of $P2$, it cannot contain a subsequence of the form $\beta_{2}, \beta_{1}, \beta_{2}$.
It follows that $w_{1}, \ldots, w_{2k-1}$ is a sequence of $\beta_{2}$'s ending with two or fewer occurrences of $\beta_{1}$.

Now suppose that $w_{1} = \beta_{1}$. If $w_{1}, \ldots, w_{2k-1}$ begins with three or more occurrences of $\beta_{1}$, then we create a subword
from $P1$, an impossibility. Thus, either $w_{3}$ or $w_{5}$ is in $\{ \beta_{2}, \alpha_{2} \}$, and the remainder of the sequence takes one
of the previously-discussed forms.

We have now completely described the possibilities for $w_{1}, \ldots, w_{2k-1}$. Our discussion shows that $w_{1}, \ldots, w_{2k-1}$
is essentially a constant sequence of $\alpha_{2}$'s or of $\beta_{2}$'s, with a small amount of variation possible  at the beginning
and end. It follows that the number of such sequences is bounded by a constant that is independent of $k$. A similar analysis establishes
a similar form for the sequence $w_{2}, w_{4}, \ldots$. It follows that the number of $\mathcal{U}$-bad words of the form
$t w_{1} t \ldots t w_{m} t$ is bounded, by a bound that is independent of $m$.

A general $\mathcal{U}$-bad word has the form $(w_{0})tw_{1} t \ldots t w_{m} t (w_{m+1})$, for some $w_{0}, w_{m+1} \in T$, 
and it follows immediately that the number of such words is uniformly bounded, regardless of $m$. Theorem \ref{thm:intgrowth} now implies that
$G(A)$ has subexponential growth.
\end{proof}

\subsection{The Case of $0(011)^{\omega}$} \label{subsection:third}

We now consider the automaton $A$ with kneading sequence $0(011)^{\omega}$. The graph $\Gamma_{A}$ appears in Figure \ref{figure:3}, which
also indicates our convention for labeling the states.

\begin{figure} [!h]
\begin{center}
\includegraphics{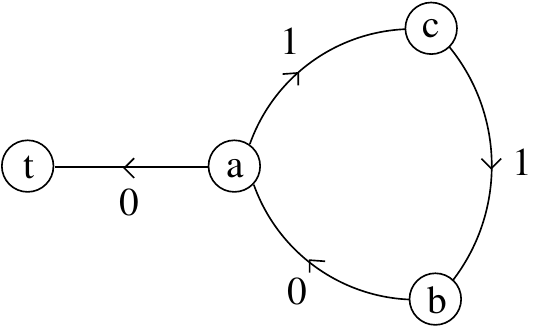}
\end{center}
\caption{The graph $\Gamma_{A}$ for the automaton $A$ with kneading sequence $0(011)^{\omega}$.}
\label{figure:3}
\end{figure}

\begin{lemma}
$\langle a, b, c \rangle \cong \langle a, b \rangle \times \langle c \rangle \cong D_{8} \times \mathbb{Z}/2\mathbb{Z}$,
where $D_{8}$ is the dihedral group of order $16$.
\end{lemma}

\begin{proof}
First, we note that $a^{2} = (t^{2}, c^{2}) = (1, c^{2})$;
$c^{2} = (1, b^{2})$; $b^{2} = (a^{2}, 1)$. It follows from this that
the automorphism $a^{2}$ has the following inductive definition: 
$a^{2} \cdot n_{1}n_{2} \ldots n_{k} = n_{1}n_{2} \ldots n_{k}$ if $n_{1}n_{2}n_{3} \neq 110$
(or if $k \leq 3$), and $a^{2} \cdot 110 n_{4} \ldots n_{k} = 110 \cdot (a^{2} \cdot n_{4} \ldots n_{k}$.
It follows directly that $|a| = 2$. Therefore, $|b| = 2$ and $|c| = 2$ by the above computations.

It follows immediately that $\langle a, b \rangle$ is a dihedral group. One can easily check that
$(ta)^{2} = (ct, tc)$, $(ta)^{4} = ((b,b),(b,b))$, and
(therefore) $(ta)^{8} = 1$. Now
$$ (ab)^{n} = ( (ta)^{n}, c^{n} ),$$
so $(ab)^{8} = 1$. We've shown that $(ta)^{4} \neq 1$, so $(ab)^{4} \neq 1$, implying that $|ab| = 8$. 
Thus, $\langle a, b \rangle \cong D_{8}$.

Next we show that $c$ commutes with $a$ and $b$. The relation $bc = cb$ follows because $c$ and $b$ have 
disjoint support, and
$$ ac = (t,cb) = (t,bc) = ca.$$

Finally, we show that $\langle c \rangle \cap \langle a, b \rangle = 1$. It is enough to show that
$c \notin \langle a, b \rangle$. The elements of $\langle a, b \rangle$ all have the form $(\underline{\hspace{8pt}}, c)$
or $(\underline{\hspace{8pt}}, 1)$. Since $c = (1,b)$ and $b \neq c$, $c \notin \langle a, b \rangle$. It follows
that $\langle c \rangle \cap \langle a, b \rangle = 1$, so $\langle a, b, c \rangle \cong \langle a, b \rangle \times \langle c \rangle$,
as claimed.
\end{proof}

We define $\widehat{\ell}: A - \{ id \} \rightarrow \mathbb{R}^{+}$ by the rule
$$ \widehat{\ell}(a) = 7; \, \, \widehat{\ell}(b) = 7; \, \, \widehat{\ell}(c) = 6 \,\, \widehat{\ell}(t) = 3.$$
We set 
$$T = \{ 1, \alpha_{1}, \alpha_{2}, \ldots, \alpha_{8}, \beta_{1}, \ldots, \beta_{7} \} \cup \{ c, c\alpha_{1}, 
\ldots, c\alpha_{8}, c\beta_{1}, \ldots, c\beta_{7} \}.$$
It is straightforward to check that $T$ satisfies the conditions of Definition \ref{def:T}.

We list the words $\alpha_{1}, \alpha_{2}, \ldots, \alpha_{8}, \beta_{1}, \ldots, \beta_{7}$, their first-line productions,
and the corresponding weights $|t| + |w|$ and $|w_{0}| + |w_{1}|$ in the table below. 
The first-line productions of the remaining words
are obtained from the entries in the table simply by post-multiplying the second coordinates by $b$. 
Similarly, the weights $|t| + |w|$ and $|w_{0}| + |w_{1}|$ can be obtained by adding (respectively)
$6$ and $7$ to the totals below. It follows easily that the length function $\ell$ is admissible.

\begin{figure} [!h]
\begin{center}
\begin{tabular}{|c|c|c|c|c|c|c|c|} \hline
$w$&$(w_{0},w_{1})$&$|t|+|w|$&$|w_{0}|+|w_{1}|$&$w$&$(w_{0},w_{1})$&$|t|+|w|$&$|w_{0}|+|w_{1}|$\\ \hline \hline
$\alpha_{1}$ & $(t,c)$ & $10$ & $9$ & $\beta_{1}$ & $(a,1)$ & $10$ & $7$ \\ \hline
$\alpha_{2}$ & $(ta,c)$ & $17$ & $16$ & $\beta_{2}$ & $(at,c)$ & $17$ & $16$ \\ \hline 
$\alpha_{3}$ & $(tat,1)$ & $24$ & $13$ & $\beta_{3}$ & $(ata,c)$ & $24$ & $23$ \\ \hline
$\alpha_{4}$ & $(tata,1)$ & $31$ & $20$ & $\beta_{4}$ & $(atat,1)$ & $31$ & $20$ \\ \hline
$\alpha_{5}$ & $(tatat,c)$ & $38$ & $29$ & $\beta_{5}$ & $(atata,1)$ & $38$ & $27$ \\ \hline
$\alpha_{6}$ & $(tatata,c)$ & $45$ & $36$ & $\beta_{6}$ & $(atatat,c)$ & $45$ & $36$ \\ \hline
$\alpha_{7}$ & $(tatatat,1)$ & $52$ & $33$ & $\beta_{7}$ & $(atatata,c)$ & $52$ & $43$ \\ \hline
$\alpha_{8}$ & $((ta)^{4},1)$ & $59$ & $40$ &         &             &    &    \\ \hline
\end{tabular}
\end{center}
\end{figure}

\begin{proposition}
The following families of words in the generators $\{ a, b, c, t \}$ are good, where
each box $\Box$ represents an occurrence of a string from $\{ c\alpha_{1}, c\alpha_{2}, c\beta_{2}, c\beta_{3} \}$:
\begin{enumerate}
\item[P0.] $twt$, $w \in T - \{ c\alpha_{1}, c\alpha_{2}, c\beta_{2}, c\beta_{3} \}$;
\item[P1.] $tc\alpha_{1}t \Box tc\alpha_{1}t$; 
\item[P2.] $tc\alpha_{1}t \Box tc\alpha_{2}t$;
\item[P3.] $t \Box t \Box t c\alpha_{1} t \Box t c\beta_{2} t \Box t \Box t$;
\item[P4.] $t \Box t \Box t c\alpha_{2} t \Box t c\alpha_{1} t \Box t \Box t$;
\item[P5.] $t c\alpha_{2} t \Box t c\beta_{2} t$;
\item[P6.] $t c\alpha_{2} t \Box t c\beta_{3} t$;
\item[P7.] $t c\beta_{2} t \Box t c\alpha_{1} t$;
\item[P8.] $t c\beta_{2} t \Box t c\alpha_{2} t$;
\item[P9.] $t \Box t \Box t c\beta_{2} t \Box t c\beta_{3} t \Box t \Box t$;
\item[P10.] $t \Box t \Box t c \beta_{3} t \Box t c\alpha_{2} t \Box t \Box t$;
\item[P11.] $t c\beta_{3} t \Box t c\beta_{2} t$;
\item[P12.] $t c\beta_{3} t \Box t c\beta_{3} t$.
\end{enumerate}
The union $\mathcal{U}$ of the above families is finite.
\end{proposition}

\begin{proof}
We check that each family is made up of good words:
\begin{enumerate}
\item[P0.] In this case, one can check that $tw$ is a reducing word in each $twt$, so the members of this family are good at depth $0$.
\item[P1.] The first-line production of each word $w$ in this family has the form $( \underline{\hspace{8pt}}, tc\beta_{1}t)$. Thus
$t c\beta_{1} t \preccurlyeq w$, so each $w$ is good at depth $1$, since $t c\beta_{1} t \in P0$ is good at depth $0$.
\item[P2.] The first-line production of each word $w$ in this family has the form $( \underline{\hspace{8pt}}, tc\beta_{1}ta)$. Thus
$t c\beta_{1} t \preccurlyeq w$, so each $w$ is good at depth $1$, since $t c\beta_{1} t \in P0$ is good at depth $0$.
\item[P3.] Let $w \in P3$. We consider the first line production $(w_{0}, w_{1})$ of $w$. If the first $\Box$ in $w$ is filled by
either $c\alpha_{1}$ or $c\beta_{2}$, then $w_{1}$ contains a copy of $t c\beta_{1} t$, and so $w$ is good at depth $1$.
If the final $\Box$ in $w$ is filled with either $c\alpha_{1}$ or $c\alpha_{2}$, then $w_{1}$ again contains a copy
of $t c\beta_{1} t$, which makes $w$ good at depth $1$. In all other cases, $w_{1}$ contains
a copy of $t c\alpha_{2} t c\beta_{2} t c\beta_{2} t$, which is good at depth $1$ (see $P5$). Therefore, in this last case,
$w$ is good at depth $2$.
\item[P4.] Let $w \in P4$. We again consider $w_{1}$ in the first line production of $w$. If the first $\Box$ in $w$ is filled
by either $c\alpha_{1}$ or $c\beta_{2}$, then $w_{1}$ contains a copy of $t c\beta_{1}t$, and so $w$ is good at depth $1$ (see $P0$).
If the final $\Box$ in $w$ is filled by either $c\alpha_{1}$ or $c\alpha_{2}$, then $w_{1}$ contains a copy
of $t c\beta_{1} t$, so again $w$ is good at depth $1$. In every other case, $w_{1}$ contains a copy
of $t c \alpha_{2} t c\alpha_{2} t c\beta_{2} t \in P5$. Since the latter word is good at depth $1$, $w$ is good at depth $2$.
\item[P5.] All words $w \in P5$ produce the word $w_{1} = t c \alpha_{3} t$ on the first line. Since the latter word is good at depth
$0$ (see $P0$), we conclude that each $w \in P5$ is good at depth $1$.
\item[P6.] All words $w \in P6$ produce the word $w_{1} = t c \alpha_{3} ta$. Since $w_{R}$ contains the protected subword
$t c \alpha_{3} t$, we conclude that each $w \in P6$ is good at depth $1$.
\item[P7.] All words $w \in P7$ produce the word $w_{1} = at c\beta_{1}t$ on the first line. Since the protected subword
$tc\beta_{1}t$ of $w_{1}$ is in $P0$, we conclude that $w$ is good at depth $1$.
\item[P8.] All words $w \in P8$ produce $w_{1} = at c\beta_{1}ta$ on the first line. It follows that
$w$ is good at depth $1$.
\item[P9.] Let $w \in P9$. If the first $\Box$ is filled by an occurrence of $c \alpha_{2}$ or $c \beta_{3}$, then
$w_{1}$ contains an occurrence of $t c \alpha_{3} t$, which makes $w$ good at depth $1$. If the last $\Box$ is filled
by an occurrence of $c\beta_{2}$ or $c \beta_{3}$, then $w_{1}$ again contains an occurrence of $t c \alpha_{3} t$, making
$w$ good at depth $1$. In all other cases, $w_{1}$ contains an occurrence of
$t c\beta_{2} t c\beta_{2} t c \alpha_{2} t \in P8$, which makes $w$ good at depth $2$.
\item[P10.] Let $w \in P10$. If the first $\Box$ is filled by an occurrence of $c \alpha_{2}$ or $c \beta_{3}$, then
$w_{1}$ contains an occurrence of $t c \alpha_{3} t$, which makes $w$ good at depth $1$. If the last $\Box$ is filled
by an occurrence of $c\beta_{2}$ or $c \beta_{3}$, then $w_{1}$ again contains an occurrence of $t c \alpha_{3} t$, making
$w$ good at depth $1$. In all other cases, $w_{1}$ contains an occurrence of
$t c\beta_{2} t c\alpha_{2} t c \alpha_{2} t \in P8$, which makes $w$ good at depth $2$. 
\item[P11.] All words $w$ in this family produce $w_{1} = atc\alpha_{3}t$ on the first line. It follows that
all words $w \in P11$ are good at depth $1$, since $tc\alpha_{3}t \in P0$ is good at depth $0$.
\item[P12.] All words $w$ in this family produce $w_{1} = atc\alpha_{3}ta$ on the first line. It follows that
all words $w \in P12$ are good at depth $1$, since $tc\alpha_{3}t \in P0$ is good at depth $0$.
\end{enumerate}
Finally, it is clear that each of these families is finite, so their union $\mathcal{U}$ is finite.
\end{proof}

\begin{theorem}
The group $G(A)$ determined by the automaton $A$ with kneading sequence $0(011)^{\omega}$
has subexponential growth. The group $G(A)$ is the iterated monodromy group of a complex post-critically finite quadratic polynomial. 
\end{theorem}

\begin{proof}
We prove the second statement first. By Theorem \ref{thm:kneadingsequence}, it is enough to show that
the automaton is planar (the other conditions being obvious). The planarity of $A$ follows from the equalities
$(tacbtacb)_{\mid 0} = cbta$
and
$(tacbtacb)_{\mid 1} = tacb$, both valid in $(A-\{ id \})^{\ast}$.

We turn to a proof of the first statement. It is sufficient to 
show that there is $M>0$ such that, for any $L>0$, there are at most $M$ $\mathcal{U}$-bad reduced words
of length exactly $L$.

We first consider reduced words of the form $tw_{1}tw_{2}t \ldots tw_{m}t$. Indeed, it is
sufficient to consider words of this form, since a general reduced word has the form $(w_{0})tw_{1}t \ldots t w_{m}t (w_{m+1})$,
and such a word is $\mathcal{U}$-bad if and only if $tw_{1}t \ldots t w_{m}t$ is. Thus, the total number of bad words of the form
$(w_{0})tw_{1}t \ldots t w_{m} t (w_{m+1})$ is a constant multiple of the number of bad words of the form
$tw_{1}t\ldots t w_{m} t$.

As in the proof of Theorem \ref{thm:110growth}, we will consider the sequences $w_{1}, w_{3}, w_{5}, \ldots$
and $w_{2}, w_{4}, w_{6}, \ldots$. The description of subfamily $P0$ shows that each $w_{i}$ ($i=1, \ldots, m$)
must be taken from $\{ c\alpha_{1}, c\alpha_{2}, c\beta_{2}, c\beta_{3} \}$. The descriptions 
of the remaining subfamilies $P1 - P12$ 
show
(essentially; see below) that certain words from $\{ c\alpha_{1}, c\alpha_{2}, c\beta_{2}, c\beta_{3} \}$ must not be followed
by certain other words from $\{ c\alpha_{1}, c\alpha_{2}, c\beta_{2}, c\beta_{3} \}$ in the sequences $w_{1}, w_{3}, \ldots$
and $w_{2}, w_{4}, \ldots$. Thus, for instance, the description of $P1$ implies that $c\alpha_{1}$ cannot follow $c\alpha_{1}$
in $w_{1}, w_{3}, \ldots$ or $w_{2}, w_{4}, \ldots$.

The exceptions are $P3$, $P4$, $P9$, and $P10$. These subfamilies also forbid one word from following another, except 
possibly at the immediate end
or beginning of $w$. We can ignore this distinction for the sake of this argument, however, since 
these exceptions allow only a finite amount of
variation at the end and beginning of $w$, and this variation will simply increase the uniform bound $M$.

With this understanding, we can make the following observations. In $w_{1}, w_{3}, \ldots$ and $w_{2}, w_{4}, \ldots$
\begin{enumerate}
\item $c\alpha_{1}$ can be followed only by $c\beta_{3}$;
\item $c\alpha_{2}$ can be followed only by $c\alpha_{2}$;
\item $c\beta_{2}$ can be followed only by $c\beta_{2}$;
\item $c\beta_{3}$ can be followed only by $c\alpha_{1}$.
\end{enumerate}
Thus, modulo the above considerations, the only possibilities for the sequences $w_{1}, w_{3}, \ldots$
and $w_{2}, w_{4}, \ldots$ are those that alternate between $c\alpha_{1}$ and $c\beta_{3}$
and constant sequences of either $c\alpha_{2}$'s or $c\beta_{2}$'s. Thus the number of
possible sequences $tw_{1}t \ldots t w_{m} t$ is bounded above, by a constant independent of $m$. The existence of the uniform bound $M$ now follows
easily, and 
Theorem \ref{thm:intgrowth} establishes that $G(A)$ has subexponential growth. 
\end{proof}

\section{A group $G(A)$ with no admissible length function} \label{section:final}

We consider the kneading automaton with kneading sequence $01(10)^{\omega}$. We label the active state $t$, and label the remaining states
$a$, $b$, and $c$, in the order that they are encountered while tracing directed edges backward from the active state in the Moore diagram.
Thus $a = (t,1)$, $b=(c,a)$, and $c=(1,b)$. Our goal in this section is to sketch a proof that $G(A)$ has no admissible length function.

\begin{proposition}
The group $G(A)$ admits no admissible length function.
\end{proposition}

\begin{proof}
Choose an arbitrary $\widehat{\ell}: A - \{ id \} \rightarrow \mathbb{R}^{+}$ and an arbitrary $T$ satisfying the conditions
from Definition \ref{def:T}. 

It turns out that: (1) the word $c$ must be in $T$; (2) one of the words $cbacb$, $cbcab$ must be in $T$, and
(3) one of the words $baba$, $abab$ must be in $T$. To prove this, 
it helps to use the homomorphism $\phi: \langle a, b, c \rangle \rightarrow \left( \mathbb{Z}/2\mathbb{Z} \right)^{3}$,
where $\phi(a) = (1,0,0)$, $\phi(b) = (0,1,0)$, and $\phi(c) = (0,0,1)$. One establishes that $\phi$ is well-defined as follows. The
subgroup $N= \langle abab, bcbc \rangle$ is central in $\langle a, b, c \rangle$, any two of the generators $a$, $b$, $c$ commute modulo $N$, and
the set $\{ 1, a, b, c, ab, ac, bc, abc \}$ is a transversal for $N$ in $\langle a, b, c \rangle$. The existence of $\phi$ now follows directly from the
First Isomorphism Theorem. One proves (1), (2), and (3) by first arguing that every other representative $w'$ of the word $w$ in question must have
at least as many occurrences of each letter as $w$, and then arguing that no other permutations of the letters of $w$ can represent the same group element.
We omit the details.     

Suppose, without loss of generality, that $\{ c, baba, cbacb \} \subseteq T$, and assume that the length function 
$\ell: (A-\{ id \})^{\ast} \rightarrow \mathbb{R}^{+} \cup \{ 0 \}$ is admissible. The first-line production of $cbacb$ is $(ctc, baba)$, so
\begin{align*}
|t| + 2|c| + 2|b| + |a| &\geq |t| + 2|a| + 2|b| + 2|c|,
\end{align*}
from which we conclude that $|a| \leq 0$. This is a contradiction.
\end{proof}


\bibliography{growth2}
\nocite{*}
\bibliographystyle{plain}

\end{document}